\documentclass[12pt,leqno]{article}
\usepackage{amsfonts}
\usepackage{amsmath, amsthm, amsfonts, amssymb, color}
\usepackage{mathrsfs}
\newtheorem{thm}{Theorem}[section]
\newtheorem{cor}[thm]{Corollary}
\newtheorem{lem}[thm]{Lemma}

\theoremstyle{definition}

\newcommand{\scr}[1]{\mathscr #1}
\definecolor{wco}{rgb}{0.5,0.2,0.3}

\numberwithin{equation}{section} \theoremstyle{remark}

\newcommand{\ua}{\uparrow}

\title{{\bf Bismut Formulae and Applications for Functional SPDEs}\footnote{Supported in
 part by NNSFC(11131003), SRFDP and the Fundamental Research Funds for the Central Universities.}
}
\author{
{\bf  Jianhai Bao$^{b)}$,  Feng-Yu Wang$^{a),b)}$,  Chenggui Yuan$^{b)}$}\\
\footnotesize{$^{a)}$School of Mathematical Sciences, Beijing Normal
University, Beijing 100875, China}\\
 \footnotesize{$^{b)}$Department of Mathematics,
Swansea University, Singleton Park, SA2 8PP, UK}\\
\footnotesize{wangfy@bnu.edu.cn, F.-Y.Wang@swansea.ac.uk,
C.Yuan@swansea.ac.uk}}
\begin{document}
\def\R{\mathbb R}  \def\ff{\frac} \def\ss{\sqrt} \def\B{\mathbf
B}
\def\N{\mathbb N} \def\kk{\kappa} \def\m{{\bf m}}
\def\dd{\delta} \def\DD{\Delta} \def\vv{\varepsilon} \def\rr{\rho}
\def\<{\langle} \def\>{\rangle} \def\GG{\Gamma} \def\gg{\gamma}
  \def\nn{\nabla} \def\pp{\partial} \def\EE{\scr E}
\def\d{\text{\rm{d}}} \def\bb{\beta} \def\aa{\alpha} \def\D{\scr D}
  \def\si{\sigma} \def\ess{\text{\rm{ess}}}
\def\beg{\begin} \def\beq{\begin{equation}}  \def\F{\scr F}
\def\Ric{\text{\rm{Ric}}} \def\Hess{\text{\rm{Hess}}}
\def\e{\text{\rm{e}}} \def\ua{\underline a} \def\OO{\Omega}  \def\oo{\omega}
 \def\tt{\tilde} \def\Ric{\text{\rm{Ric}}}
\def\cut{\text{\rm{cut}}} \def\P{\mathbb P} \def\ifn{I_n(f^{\bigotimes n})}
\def\C{\scr C}      \def\aaa{\mathbf{r}}     \def\r{r}
\def\gap{\text{\rm{gap}}} \def\prr{\pi_{{\bf m},\varrho}}  \def\r{\mathbf r}
\def\Z{\mathbb Z} \def\vrr{\varrho} \def\ll{\lambda}
\def\L{\scr L}\def\Tt{\tt} \def\TT{\tt}\def\II{\mathbb I}
\def\i{{\rm in}}\def\Sect{{\rm Sect}}\def\E{\mathbb E} \def\H{\mathbb H}
\def\M{\scr M}\def\Q{\mathbb Q} \def\texto{\text{o}} \def\LL{\Lambda}
\def\Rank{{\rm Rank}} \def\B{\scr B} \def\i{{\rm i}} \def\HR{\hat{\R}^d}
\def\to{\rightarrow}\def\l{\ell}
\def\8{\infty}

\maketitle

\begin{abstract} By using Malliavin calculus, explicit derivative formulae are established for a class of
semi-linear functional stochastic partial differential equations
with additive or multiplicative noise. As applications, gradient
estimates and Harnack inequalities are derived for the semigroup of
the associated segment process.

\end{abstract}
 \vskip 2cm

\section{Introduction}
The Bismut-type formulae, initiated in \cite{Bismut}, are  powerful tools to
derive regularity estimates for the underlying Markov semigroups.  The formulae have been developed and applied
in various settings,
e.g., in \cite{dz96} for stochastic partial differential equations
(SPDEs) driven by cylindrical Wiener processes and \cite{dx10} for
semi-linear SPDEs with L\'{e}vy noise, using a simple martingale
approach proposed by Elworthy-Li \cite{el94}; in \cite{W11b} for
linear stochastic differential equations (SDEs) driven by (purely jump) L\'{e}vy
processes in terms of lower bound conditions of   L\'{e}vy measures; in
\cite{bwy11,gw11} for degenerate SDEs with additive noise, using a
coupling technique;   in \cite{f96,p06,wz11,z10} for degenerate
SDEs    using Malliavin calculus.

 However, there are few analogues for functional SPDEs (even for finite-dimensional functional SDEs) with
multiplicative noise.  In this paper we aim to establish explicit Bismut-type formulae for a
class of functional SPDEs with additive or multiplicative noise. Noting that   for functional SDEs  the martingale method
used in \cite{el94} does not work due to the lack of     backward
Kolmogorov equation for the segment process, and  the coupling method
developed in \cite{ATW06,bwy11,gw11, wx10} seems not easy to apply provided the noise is multiplicative,
we will  mainly make use of   Malliavin calculus.

 Let $(H,\langle\cdot,\cdot\rangle,\|\cdot\|)$ be a real
separable Hilbert space,  and  $(W(t))_{t\geq0}$   a cylindrical
Wiener process on $H$ with respect to  a complete probability space
$(\Omega, \scr {F}, \mathbb{P})$ with the natural filtration
$\{\scr {F}_t\}_{t\geq0}$. Let $\scr {L}(H)$ and $\scr
{L}_{HS}(H)$ be the spaces of all linear bounded operators and
Hilbert-Schmidt operators on $H$ respectively. Denote by $\|\cdot\|$
and $\|\cdot\|_{HS}$  the operator norm and the Hilbert-Schmidt norm
respectively.  Let $\tau>0$ be fixed and let
$\mathscr{C}:=C([-\tau,0]\rightarrow H)$, the space of all
$H$-valued continuous functions   defined on $[-\tau,0]$, equipped with the
uniform norm
$\|f\|_\infty:=\sup_{-\tau\leq\theta\leq0}\|f(\theta)\|$. For a
map $h:[-\tau,\infty)\rightarrow H$ and $t\geq0$, let
$h_t\in\mathscr{C}$ be the segment of $h(t)$, i.e.
$h_t(\theta)=h(t+\theta),\theta\in[-\tau,0]$.

Consider the following semi-linear functional SPDE
\begin{equation}\label{eq1}
\begin{cases}
\d X(t)=\{AX(t)+F(X_t)\}\d t+\sigma(X(t))\d W(t),\\
X_0=\xi\in\mathscr{C},
\end{cases}
\end{equation}
 where
\begin{enumerate}
\item[\textmd{(A1)}] $(A,\scr {D}(A))$ is a linear
operator on $H$ generating a contractive $C_0$-semigroup
$(\e^{tA})_{t\geq0}.$
\item[\textmd{(A2)}]
$F:\mathscr{C}\rightarrow H$ is Fr\'{e}chet differentiable such that
$\nabla F:\C\times\C\to H$ is bounded on $\C\times\C$ and   uniformly continuous on bounded sets.
\item[\textmd{(A3)}] $ \sigma: H\rightarrow \scr {L}(H)$
is Fr\'{e}chet differentiable such that $\nabla \sigma: H\times H\to \L_{HS}(H)$ is bounded on $H\times H$ and  uniformly continuous on bounded sets, and $\si(x)$ is invertible for each $x\in H$.
\item[\textmd{(A4)}]   $\int_0^t s^{-2\alpha}\|e^{s A}\si(0)\|_{HS}^2\d s<\infty$ holds for some constant $\aa\in (0,\frac{1}{2})$ and all $t>0.$
\end{enumerate}

Recall that a mild solution is a continuous adapt process $(X(t))_{t\ge -\tau}$ on $H$ such that
$$X(t)= \e^{t A}\xi(0) + \int_0^t \e^{(t-s)A} F(X_s)\d s +\int_0^t\e^{(t-s)A}\si(X(s))\d W(s),\ \ t\ge 0.$$
 By $(A1)-(A4)$, equation \eqref{eq1} has a unique mild solution (see
Theorem  A.1 in the Appendix section), denoted by
$(X^\xi(t))_{t\geq0}$, the solution with
$X_0=\xi\in\mathscr{C}$. Let
\begin{equation*}
P_tf(\xi):=\mathbb{E}f(X_t^\xi), \ \ \ \  t\geq0,
 \xi\in\mathscr{C}, f\in\mathscr{B}_b(\mathscr{C}),
\end{equation*}
where $\mathscr{B}_b(\mathscr{C})$ is the class of all bounded
measurable functions on $\mathscr{C}$. We remark that due to the
time-delay the solution $(X^\xi(t))_{t\geq0}$ is not Markovian, but
its segment process $(X^\xi_t)_{t\geq0}$ admits strong Markov
property, so that $P_t$  is a Markov semigroup on
$\mathscr{B}_b(\mathscr{C})$.

The following two theorems are the main results of the paper, which
provide derivative formulae for $P_t$ with additive and
multiplicative noise respectively.

\begin{thm}[Additive Noise]\label{T1.1} Assume that $(A1)$-$(A4)$ hold with
constant $\sigma\in\scr {L}(H)$.
Then for any $T>\tau$ and $u\in C^1([0,\infty))$ such that $u(0)=1$
and $u(t)=0$ for $t\ge T-\tau$,
\begin{equation}\label{eq3}
\nabla_\eta
P_Tf(\xi)=\mathbb{E}\bigg(f(X_T^\xi)\int_0^{T}\big\langle
\sigma^{-1}(\nabla_{\Upsilon_t}F(X_t^\xi)-\dot{u}(t)\e^{tA}\eta(0)),\d
W(t)\big\rangle\bigg)
\end{equation}
holds for all $\xi,\eta\in\mathscr{C}$ and $f\in
C^1_b(\mathscr{C})$, where
\begin{equation*}
\Upsilon(t):=
\begin{cases}
u(t)\e^{tA}\eta(0), \ & t>0,\\
 \eta(t),\  &  t\in[-\tau,0].
\end{cases}
\end{equation*}
\end{thm}

\begin{thm}[Multiplicative Noise]\label{T1.2}
  Assume that $(A1)$-$(A4)$ hold. Let $T>\tau$ and $u\in
C^1([0,\infty))$ be such that $u(t)>0$ for $t\in[0,T-\tau)$,
$u(t)=0$ for $t\ge T-\tau,$ and  $$\theta_p:=\inf_{t\in[0,T-\tau]}\big\{p+(p-1)u'(t)\big\}>0$$ holds for some $p>1$.  Then
for any $\xi,\eta\in\mathscr{C}$:
\begin{enumerate}
\item[$(1)$] The equation
\begin{equation}\label{eq4}
\begin{cases}
\d Z(t)=\Big\{AZ(t)+\big(\nabla_{Z_t}F(X_t^\xi)-\frac{Z(t)}{u(t)}\big)1_{[0,T-\tau)}(t)\Big\}\d t\\
\ \ \ \ \ \ \ \ \ \ \ \ \  +(\nabla_{Z(t)}\sigma(X^\xi(t)))\d W(t),\\
 Z_0=\eta,
\end{cases}
\end{equation}
has a unique solution such that $Z(t)=0$ for $t\geq T-\tau$.
\item[$(2)$] If    $\|\si^{-1}(\cdot)\|\le c(1+\|\cdot\|^q)$ holds for some constants $c,q>0$, then
\begin{equation}\label{eq5}
\begin{split}
\nabla_\eta
P_Tf(\xi)&=\mathbb{E}\bigg(f(X_T^\xi)\int_0^{T}\Big\langle
\sigma^{-1}(X^\xi(t))\Big\{\frac{Z(t)}{u(t)}1_{[0,T-\tau)}(t)\\
&\ \ \ \ \ \ \ \ +\nabla_{Z_t}F(X_t^\xi)1_{[T-\tau,T]}(t)\Big\},\d
W(t)\Big\rangle\bigg)
\end{split}
\end{equation}
holds for $f\in C^1_b(\mathscr{C})$.
\end{enumerate}
\end{thm}

A simple choice of $u$ for Theorem \ref{T1.1} is $u(t)=\ff{(T-\tau-t)^+}{T-\tau},$ while for Theorem \ref{T1.2} one may take $u(t)=(T-\tau-t)^+$ such that $\theta_p=1$ for all $p>1.$
Both theorems will be proved in the next section. In Section 3 these results are
  applied to derive
explicit gradient estimates and Harnack inequalities of $P_t$. Finally, for completeness, in the Appendix section we address the existence and
uniqueness of mild solution to equation \eqref{eq1} under
 (A1)-(A4), and the existence of Malliavin derivative $D_hX^\xi(t)$
along direction $h$ and derivative process $\nabla_{\eta}X^\xi(t)$
along direction $\eta$ as solutions of SPDEs on $H$.

\section{Proofs of Theorems \ref{T1.1} and \ref{T1.2}}
For the readers' convenience, let us first explain the main idea of
establishing Bismut formula using Malliavin calculus. Let $H^1_a$ be
the class of all adapted process $h=(h(t))_{t\geq0}$ on $H$ such
that $h(0)=0$,
\begin{equation*}
\dot{h}(t):=\frac{\d}{\d t}h(t)
\end{equation*}
exists $\mathbb{P}\times \d t$-a.e. and
\begin{equation*}
\mathbb{E}\int_0^T\|\dot{h}(t)\|^2\d t<\infty,\ \ \ \ T>0.
\end{equation*}
For $\epsilon>0$ and $h\in H^1_a$, let $X^{\xi,\epsilon h}(t)$ solve
\eqref{eq1} with $W(t)$  replaced by $W(t)+\epsilon h(t)$, i.e.,
\begin{equation}\label{eq6}
\begin{cases}
\d X^{\xi,\epsilon h}(t)=\{AX^{\xi,\epsilon h}(t)+F(X^{\xi,\epsilon h}_t)+\epsilon\sigma(X^{\xi,\epsilon h}(t))\dot{h}(t)\}\d t\\
\ \ \ \ \ \ \ \ \ \ \ \ \ \ \ \ \ +\sigma(X^{\xi,\epsilon h}(t))\d W(t),\\
X_0^{\xi,\epsilon h}=\xi\in\mathscr{C}.
\end{cases}
\end{equation}
If for $h\in H^1_a$
\begin{equation*}
D_hX_t^\xi:=\frac{\d}{\d\epsilon}X_t^{\xi,\epsilon
h}\Big|_{\epsilon=0}
\end{equation*}
exists in $L^2(\Omega\rightarrow H;\mathbb{P})$, we call it the
Malliavin derivative of $X_t^\xi$ along direction $h$. Next, let
\begin{equation*}
\nabla_\eta X_t^\xi:=\frac{\d}{\d\epsilon}X_t^{\xi+\epsilon\eta
}\Big|_{\epsilon=0}
\end{equation*}
be the derivative process of $X_t^\xi$ along direction
$\eta\in\mathscr{C}$. If
\begin{equation}\label{eq7}
D_hX_T^\xi=\nabla_\eta X_T^\xi, \ \ \mbox{ a.s., }
\end{equation}
then for any $f\in C_b^1(\mathscr{C})$
\begin{equation*}
\begin{split}
\nabla_\eta P_Tf(\xi)&=\mathbb{E}\nabla_\eta f(X_T^\xi)
=\mathbb{E}\nabla_{\nabla_\eta X_T^\xi}f(X_T^\xi)\\
&=\mathbb{E}\nabla_{D_hX_T^\xi}f(X_T^\xi)=\mathbb{E}D_hf(X_T^\xi).
\end{split}
\end{equation*}
Combining this with the integration by parts formula for $D_h$, we
obtain
\begin{equation*}
\nabla_\eta P_Tf(\xi)=\mathbb{E}\Big(f(X_T^\xi)\int_0^T\langle
\dot{h}(t),\d W(t)\rangle\Big).
\end{equation*}
In conclusion, the key point of the proof is, for given $T>\tau$,
$\xi,\eta\in\mathscr{C}$ and $f\in C_b^1(\mathscr{C})$, to construct
an $h\in H^1_a$ such that \eqref{eq7} holds.

We are now in a position to complete the proofs of Theorems
\ref{T1.1} and \ref{T1.2}.

\beg{proof}[Proof of Theorem \ref{T1.1}] Let $h(0)=0$ and
\begin{equation*}
\dot{h}(t) =\sigma^{-1}\big\{\nabla_{\Upsilon_t}F(X_t^\xi)-\dot{u}(t)\e^{tA}\eta(0)\big\},\
\ \ t\geq0.
\end{equation*}
By $(A1)$ and $u\in C^1([0,T-\tau])$, we see that $h\in H^1_a$.
Moreover, $\Upsilon(t)$ solves the equation
\begin{equation}\label{eq8}
\begin{cases}
\d\Upsilon(t)=\{A\Upsilon(t)+\nabla_{\Upsilon_t}F(X_t^\xi)-\sigma \dot{h}(t)\}\d t, \ \ t\geq0,\\
\Upsilon_0=\eta.
\end{cases}
\end{equation}
On the other hand, by Theorem A.2 in Appendix, when
$\nabla\sigma=0$, $\nabla_\eta X^\xi(t)-D_hX^\xi(t)$ also solves
this equation. Since it is trivial that \eqref{eq8} has a unique
solution, we conclude that
\begin{equation*}
\nabla_\eta X^\xi(t)-D_hX^\xi(t)=\Upsilon(t),\ \ \  t\geq0.
\end{equation*}
Thus, $\nabla_\eta X^\xi_T=D_hX^\xi_T$ as $\Upsilon_T=0$ according to
the choice of $u$. Therefore, the desired derivative formula holds
as explained above.\end{proof}

To prove Theorem \ref{T1.2}, we need the following lemma.   Since $\nn_\cdot F (X_t^\xi): \C\to H$  and $\nn_\cdot\si(X^\xi(t)): H\to \L_{HS}(H)$ are linear and bounded,
$(\ref{eq4})$ has a unique  strong (variational) solution for $t\in [0,T-\tau).$

\begin{lem}\label{LB}In the situation of Theorem $\ref{T1.2}$, let   $(Z(t))_{t\in [0,T-\tau)}$ solve $(\ref{eq4})$. Then for any $p>0$ there  exists a constant $C>0$ such that
$$\sup_{t\in [0,T-\tau)} \|Z_t\|_\infty^p < C \|\eta\|_\infty^p,\ \ \eta\in \C.$$
 \end{lem}

\begin{proof} It suffices to prove for $p>2.$  By It\^o's formula and the boundedness of $\nn F$ and $\nn\si$, there exists a constant $c_1>0$ such that
\beg{equation*}\beg{split} \d \|Z(t)\|^2 &= \Big\{2\Big\<Z(t), AZ(t)+\nn_{Z_t}F(X_t^\xi)-\ff{Z(t)}{u(t)} \Big\> +\|\nn_{Z(t)}\si(X^\xi(t))\|_{HS}^2\Big\}\d t\\
 &\qquad + 2 \<Z(t), (\nn_{Z(t)}\si(X^\xi(t))\d W(t)\>\\
 &\le \Big\{c_1\|Z_t\|_\infty^2-\ff{2\|Z(t)\|^2}{u(t)}1_{[0,T-\tau)}(t)\Big\}\d t + 2 \<Z(t), (\nn_{Z(t)}\si(X^\xi(t))\d W(t)\> \end{split}\end{equation*} holds for $t\in [0,T-\tau).$
So, for $p>2$ there exists a constant $c_2>0$ such that
\beq\label{AC1} \beg{split}&\d \|Z(t)\|^p  = \d(\|Z(t)\|^2)^{\ff p 2} \\
&= \Big\{\ff p 2 \|Z(t)\|^{p-2}\d\|Z(t)\|^2 +\ff p 2 (p-2) \|Z(t)\|^{p-4}\big\|\big(\nn_{Z(t)}\si(X^\xi(t))\big)^*Z(t)\big\|^2\Big\}\d t\\
&\qquad\qquad +p\|Z(t)\|^{p-2}\big\<Z(t), \big(\nn_{Z(t)}\si(X^\xi(t))\big)\d W(t)\big\>\\
&\le \Big\{c_2 \|Z_t\|_\infty^p -\ff{p\|Z(t)\|^p}{u(t)}\Big\}\d t +p \|Z(t)\|^{p-2} \big\<Z(t), \big(\nn_{Z(t)}\si(X^\xi(t))\big)\d W(t)\big\>\end{split}\end{equation} holds for $t\in [0,T-\tau).$ Since $\|\nn_{Z(t)}\si(X^\xi(t))\|_{HS}\le c\|Z(t)\|$ holds for some constant $c>0$, combining this with the Burkhold-Davis-Gundy inequality,
we arrive at
$$\E\sup_{s\in [-\tau, t]}\|Z(s)\|^p\le \|\eta\|_\infty^p +c_3\int_0^t \E \sup_{s\in [-\tau,\theta]}\|Z(s)\|^p\d s,\ \ t\in [0,T-\tau)$$ for some constant $c_3>0.$ The proof is then completed by the Gronwall lemma. \end{proof}

\beg{proof}[Proof of Theorem \ref{T1.2}]  (1) Due to $(A1)-(A4)$, it is
easy to see that \eqref{eq4} has a unique solution for
$t\in[0,T-\tau)$. Let
\begin{equation*}
\tilde{Z}(t)=Z(t)1_{[-\tau,T-\tau)}(t), \ \ \ \ \ t\geq-\tau.
\end{equation*}
If
\begin{equation}\label{eq9}
\lim\limits_{t\uparrow T-\tau}Z(t)=0,
\end{equation}
then it is easy to see that $(\tilde{Z}(t))_{t\geq0}$ solves
\eqref{eq4} and hence, the proof is finished.  By It\^o's formula and (\ref{AC1})  we can deduce that
\begin{equation*}
\begin{split}
\d\frac{\|Z(t)\|^p}{u^{p-1}(t)} &=\frac{1}{u^{p-1}(t)}\d\|Z(t)\|^p-(p-1)\frac{\dot{u}(t)\|Z(t)\|^p}{u^p(t)}\d t\\
&\leq-\theta_p\frac{\|Z(t)\|^p}{u^p(t)}\d t+C_1\|Z(t)\|^p_\infty\d t\\
&\quad+\frac{ p}{u^{p-1}(t)}\|Z(t)\|^{p-2} \langle
Z(t),(\nabla_{Z(t)} \sigma(X^\xi(t))) \d W(t)\rangle
\end{split}
\end{equation*}
for some constant $C_1>0$. Combining this with Lemma \ref{LB} we obtain
\begin{equation}\label{eq21}
\mathbb{E}\int_0^{T-\tau}\frac{\|Z(t)\|^p}{u^p(t)}\d t\leq
C_2\Big(\|\eta\|_\infty^p +\ff{\|\eta(0)\|^p}{u^{p-1}(0)}\Big)
\end{equation}
for some constant $C_2>0$, and due to the Burkhold-Davis-Gundy inequality $$
\mathbb{E}\sup_{s\in[0,T-\tau)}\frac{\|Z(s)\|^p}{u^{p-1}(s)}<\infty.$$
Since $u(s)\downarrow 0$ as $s\uparrow T-\tau$, the latter   implies (\ref{eq9}).

  (2) Let
\begin{equation*}
\begin{split}
h(t)=\int_0^t\sigma^{-1}(X^\xi(s))\Big\{\frac{Z(s)}{u(s)}1_{[0,T-\tau)}(s)+\nabla_{Z_s}F(X_s^\xi)1_{[T-\tau,T]}(s)\Big\}\d
s, \ \ t\geq0.
\end{split}
\end{equation*}
We first prove that $h\in
H_a^1$.
According to the boundedness of $\|\nabla F\|$ and  using the H\"older  inequality,
we obtain
\begin{equation*}
\begin{split}
&\mathbb{E}\int_0^T\|\dot{h}(t)\|^2\d t
 \leq\mathbb{E}\int_0^{T-\tau}\|\sigma^{-1}(X^\xi(t))\|^2\frac{\|Z(t)\|^2}{u^2(t)}\d
t+C\mathbb{E}\int_{T-\tau}^T\|\sigma^{-1}(X^\xi(t))\|^2
\|Z_s\|^2_\infty \d t\\
&\leq\Big(\mathbb{E}
\int_0^T\|\sigma^{-1}(X^\xi(t))\|^{\frac{2p}{p-2}}\d t\Big)^{\frac{p-2}{p}}
 \times\bigg\{
\Big(\mathbb{E}\int_0^{T-\tau}\frac{\|Z(t)\|^p}{u^p(t)}\d
t\Big)^{\frac{2}{p}}+C
\Big(\mathbb{E}\int^T_{T-\tau}\|Z_t\|^p_\infty \d
t\Big)^{\frac{2}{p}}\bigg\}
\end{split}
\end{equation*}
for some constant $C>0$. Combining this with (\ref{eq21}), $\|\si^{-1}(x)\|\le c(1+\|x\|^q)$, Lemma \ref{LB} and Theorem A.1 below,
we conclude that $\E\int_0^T\|\dot h(t)\|^2\d t<\infty$; that is, $h\in H_a^1.$

Next, we intend to show that $\nn_h X_T^\xi= D_h X_T^\xi,$ which implies the desired derivative
formula as explained in the beginning of this section. It is easy to see from Theorem A.2 below and the definition of $h$  that
$\Gamma(t):=\nabla_\eta X^\xi(t)-D_hX^\xi(t)$ solves the equation
\begin{equation*}
\begin{cases}
\d\Gamma(t)=\Big\{A\Gamma(t)+\nabla_{\Gamma_t}F(X_t^\xi)-\frac{Z(t)}{u(t)}1_{[0,T-\tau)}(t)-\nabla_{Z_t}F(X_t^\xi)1_{[T-\tau,T]}(t)\Big\}\d t\\
\ \ \ \ \ \ \ \ \ \ \quad+\nabla_{\Gamma(t)}\sigma(X^\xi(t))\d W(t),\ \ t\in [0,T]\\
 \Gamma_0=\eta.
\end{cases}
\end{equation*}
Then for $t\in [0,T],$
\begin{equation*}
\begin{cases}
\d(\Gamma(t)-Z(t))=\Big\{A(\Gamma(t)-Z(t))+\nabla_{\Gamma_t-Z_t}F(X_t^\xi)\Big\}\d t+\nabla_{\Gamma(t)-Z(t)}\sigma(X^\xi(t))\d W(t),\\
\Gamma_0-Z_0=0.
\end{cases}
\end{equation*}
By It\^o's formula and using  (A1)-(A3), we obtain
\begin{equation*}
\d\|\Gamma(t)-Z(t)\|^2\leq C\|\Gamma_t-Z_t\|^2_\infty \d
t+2\langle\Gamma(t)-Z(t),\nabla_{\Gamma(t)-Z(t)}\sigma(X^\xi(t))\d
W(t)\rangle
\end{equation*}
for some constant $C>0$ and all $t\in [0,T]$. By the boundedness of
$\|\nabla\sigma\|_{HS}$ and applying the Burkhold-Davis-Gundy
inequality, we obtain
\begin{equation*}
\mathbb{E}\sup_{s\in[0,t]}\|\Gamma(s)-Z(s)\|^2\leq
C'\int_0^t\mathbb{E}\sup_{s\in[0,r]}\|\Gamma(s)-Z(s)\|^2\d r, \ \
t\in [0,T]
\end{equation*}
for some constant $C'>0$. Therefore $\Gamma(t)=Z(t)$  for all
$t\in [0,T]$. In particular, $\Gamma_T=Z_T$. Since $Z_T=0$, we obtain    $\nabla_\eta
X_T^\xi=D_hX_T^\xi$.\end{proof}
\paragraph{Remark 2.1.}
 Our main results, Theorem \ref{T1.1} and Theorem \ref{T1.2},
are established under the assumption that the infinitesimal
generator $A$  generates a contractive $C_0$-semigroup. Replacing $A$ and $F(x)$ by $A-\aa$ and $F(x)+\aa x$ for a positive constant $\aa>0$, they also work for   $A$ generating a
pesudo-contractive $C_0$-semigroup, i.e., $\|\e^{tA}\|\leq \e^{\alpha
t}$.

\section{Gradient Estimate  and Harnack Inequality}
In this section we give some applications of Bismut formulae for
$P_t$ with additive and multiplicative noise respectively.
\begin{thm}[Additive Noise]\label{T3.1} Assume that $(A1)-(A4)$ hold with
constant $\sigma\in\scr {L}(H)$.
  Then there exists a constant $C>0$ such that

$(1)$   For any $T>\tau,\xi,\eta\in \mathscr{C}$
and $f\in\mathscr{B}_b(\mathscr{C})$,
\begin{equation*}
|\nabla_\eta P_Tf(\xi)|^2\leq \ff{C}{(T-\tau)\land 1}P_Tf^2(\xi).
\end{equation*}

$(2)$   For any $T>\tau,\xi,\eta\in \mathscr{C}$
and positive $f\in\mathscr{B}_b(\mathscr{C})$,
\begin{equation}\label{eq14}
|\nabla_\eta P_Tf(\xi)|\leq\delta\big\{P_T(f\log f)-(P_Tf)\log
P_Tf\big\}(\xi)+\frac{\|\eta\|_\infty^2}{\delta \{(T-\tau)\land 1\} } P_Tf(\xi),\ \delta>0.
\end{equation}
\end{thm}

\begin{proof} By the Jensen inequality and the semigroup property of $P_t$, it suffices to prove for $T-\tau\in (0,1].$ Let $u(t)= \ff{(T-\tau-t)^+}{T-\tau}.$
 By Theorem \ref{T1.1}, the proof is then standard and similar to that of \cite[Theorem 4.2]{gw11}. We include it below for completeness.

 \noindent (1) Note that $\dot{u}(t)=-\frac{1}{T-\tau}$. Due to the definition
of $\Upsilon(t)$ and the boundedness of $\|\nabla F\|$ it follows
that
\begin{equation*}
|\nabla_{\Upsilon_t}F(X_t^\xi)|^2\leq C\|\eta\|_\infty^2
\end{equation*}
for some constant $C>0$. By \eqref{eq3}, H\"older's inequality and
the boundedness of $\|\sigma^{-1}\|$ we have
\begin{equation}
\begin{split}
|\nabla_\eta P_Tf(\xi)|^2&\leq
2P_Tf^2(\xi)\mathbb{E}\int_0^{T}\Big\{\|
\sigma^{-1} \nabla_{\Upsilon_t}F(X_t^\xi)\|^2+\frac{1_{[0,T-\tau)}(s)}{(T-\tau)^2}\|\e^{tA}\eta(0))\|^2\Big\}\d t\\
&\leq \ff{C}{ T-\tau }\|\eta\|_\infty^2P_Tf^2(\xi)
\end{split}
\end{equation}for some constant $C>0$ and all $T\in (\tau,\tau+1]$.

\noindent (2) For $t\in[0,T]$, let
\begin{equation*}
M(t):=\int_0^t\Big\langle
\sigma^{-1}\Big(\nabla_{\Upsilon_s}F(X_s^\xi)+\frac{1_{[0,T-\tau)}(s)}{T-\tau}\e^{sA}\eta(0)\Big),\d
W(s)\Big\rangle,
\end{equation*}
which is a mean-square integrable martingale, with quadratic
variation process
\begin{equation*}
\langle M\rangle(t):=\int_0^t\Big\|
\sigma^{-1}\Big(\nabla_{\Upsilon_s}F(X_s^\xi)+\frac{1_{[0,T-\tau)}(s)}{T-\tau}\e^{sA}\eta(0)\Big)\Big\|^2\d
s\leq \ff{C\|\eta\|_\infty^2}{T-\tau},\ \ t\in [0,T]
\end{equation*}
for some constant $C>0$. In the light of \eqref{eq3} and Young's
inequality \cite[Lemma 2.4]{atw09},  we have that for any $\delta>0$
and positive $f\in \mathscr{B}_b(\scr {C})$
\begin{equation*}
|\nabla_\eta P_Tf(\xi)|\leq\delta\big\{P_T(f\log f)-(P_Tf)\log
P_Tf\big\}(\xi)+\delta
P_Tf(\xi)\log\mathbb{E}\exp\Big(\frac{1}{\delta}M(T)\Big).
\end{equation*}
Moreover, by the exponential martingale inequality, the boundedness
of $\|\nn F\|$ and the definition of $\Upsilon_s$,
\begin{equation*}
\begin{split}
\mathbb{E}\exp\Big(\frac{1}{\delta}M(T)\Big) \leq \bigg\{\E\exp\Big(\frac{2}{\delta^2}\<M\>(T)\Big)\bigg\}^{\ff 1 2}
 \leq\exp\Big(\frac{C}{\delta^2(T-\tau) }\|\eta\|_\infty^2\Big)
\end{split}
\end{equation*} holds for some constant $C>0$ and all $T\in (\tau,\tau+1].$  Therefore, the proof is finished.
\end{proof}

According to \cite[Proposition
4.1]{gw11},   \eqref{eq14} implies the following
Harnack inequality. Applications of these inequalities to heat kernel estimates, invariant probability measure and Entropy-cost inequalities can be found in e.g. \cite{RW,W10,W11b}.
\begin{cor}\label{cor3.2} Assume that $(A1)-(A4)$ hold with
constant $\sigma\in\scr {L}(H)$.   Then there exists a constant $C>0$ such that
\begin{equation}\label{eq13}
|P_Tf|^\alpha(\xi)\leq \exp\Big[ \frac{\alpha C\|\eta\|_\infty^2}
 { (\alpha-1)\{(T-\tau)\land 1\} }\Big]P_T|f|^\alpha(\xi+\eta),\ \ f\in\mathscr{B}_b(\mathscr{C}), T>\tau, \xi,\eta\in\mathscr{C}
\end{equation}
holds for any $\alpha>1$.
\end{cor}

Next, we consider the multiplicative noise case. For simplicity we only consider the case where $\|\si^{-1}\|_\infty:=\sup_{x\in H} \|\si^{-1}(x)\|<\infty$.
The case for $\si^{-1}$ having algebraic growth is similar, where the resulting estimate of   $\|\nn P_t f\|$ will be no longer bounded for bounded $f$, but bounded above by a polynomial function of $\|\xi\|_\infty$.

\begin{thm}[Multiplicative Noise]\label{T3.3} Let Assume $(A1)$-$(A4)$ and assume that $\|\si^{-1}\|_\infty<\infty$.   Then for any $p> 1$ there exists a constant $C>0$ such that

\begin{equation*}
|\nabla_\eta P_Tf(\xi) |   \leq \ff {C\|\eta\|_\infty } { 1\land
\ss{T-\tau} }  (P_T|f|^p)^{\ff 1 p} (\xi),\ \ f\in\B_b(\C), T>\tau,
\xi,\eta\in\mathscr{C}.
\end{equation*} In particular, $P_t$ is strong Feller for $t>T-\tau$.
\end{thm}

\begin{proof} It suffices to prove for $T\in (\tau,\tau+1].$   Let $u(t)=(T-\tau-t)^{+}, t\geq0$. We have $\theta_p=1$. Since $\si^{-1}$ is bounded,  for any $p>1$ and $\eta\in \C$, it follows from (\ref{eq5}) that
\beg{equation*}\beg{split}&\ff{ |\nn_\eta P_Tf|^{\ff p {p-1}}(\xi)}{(P_T|f|^p)^{\ff 1{p-1}}(\xi)}\\
 &\le  \E\bigg|\int_0^T \Big\<\si^{-1}(X^\xi(s))\Big\{\ff{Z(s)}{u(s)}1_{[0,T-\tau)}(s)+\nn_{Z_s}F(X_s^\xi)1_{[T-\tau,T]}(s)\Big\}, \d W(s)\Big\>\bigg|^{\ff p{p-1}} \\
&\le C_1 \E\bigg(\int_0^T  \Big( \ff{|Z(t)|^2}{u^2(t)}1_{[0,T-\tau)}(t) +\|Z_t\|_\infty^21_{[T-\tau,T]}(t)\Big) \d t\bigg)^{\ff{p}{2(p-1)}}\end{split} \end{equation*}holds for some constants
$C_1,C_2>0$ and all $T\in (\tau,\tau+1],$ where the second inequality follows from the Burkholder-Davis-Gundy inequality: for any $q>1$ there exists a constant $C_q>0$ such that
$$\E \sup_{t\in [0,T]} |M(t)|^q\le C_q\E\<M\>^{\ff q 2}
(T)$$ holds for any continuous martingale $M(t)$ and $T>0.$ 
Then the proof is completed by combining this with \eqref{eq21} with $u(0)=T-\tau$ and Lemma  \ref{LB}.
\end{proof}

\paragraph{Remark 3.1.}  From Corollary \ref{cor3.2} and \cite[Proposition 4.1]{gw11},
we know that entropy estimation \eqref{eq14} plays a key role in
establishing the Harnack inequality. However, the entropy estimation
 seems to be difficult to obtain for the multiplicative
noise case. Hence we can not adopt the same method as in the additive
noise case to derive the Harnack inequality. In order to establish
the Harnack inequality for the multiplicative noise case, one may
use coupling method as in Wang \cite{w11}, and Wang and Yuan
\cite{wy11}. Since the derivation of the Harnack inequality for
functional SPDEs with multiplicative noise is very similar to that
of \cite{wy11}, we omit it here.

\appendix
\section{Appendix}
In this section we  give two auxiliary lemmas, where one concerns
the existence and uniqueness of solution of equation \eqref{eq1}
under  (A1)-(A4), and the other one discusses not only the existence
of Malliavin directional derivative but also the derivative process
with respect to the initial data. To make the content
self-contained, we sketch their proofs.

\begin{thm} Let $(A1), (A4)$ hold, and let $F: H\to H, \si: H\to \L(H)$ be Lipschitz continuous.
Then for any $p>2$ and initial data
$\xi\in L^p(\OO\to \C,\F_0,\P)$, equation \eqref{eq1} has a unique mild solution
$(X^\xi(t))_{t\geq0}$, and the solution satisfies
$$\E \sup_{t\in [0,T]} \|X_t^\xi\|_\infty^p<\infty, \ \ T>0.$$
\end{thm}

\begin{proof} Obviously, (A4) remains true by replacing $\aa$ with a smaller positive number. So, we may take in  (A4)   $\aa\in (0, \ff 1 p).$ Then, by \cite[Proposition 7.9]{DZ} for $r=\ff p 2\in (1,\ff 1 {2\aa}),$ for any $T_0>0$ there exists a constant $C_0>0$ such that for any continuous adapted process $Y(s)$ on $H$,
\beq\label{BB1} \E\sup_{t\in [0,T]} \bigg|\int_0^t\e^{(t-s)A}\si(Y(s))\d W(s)\bigg|^p\le C_0 \E\int_0^T\|\si(Y(s)\|^p\d s,\ \ T\in [0,T_0].\end{equation}
Using this inequality,   the desired assertions follow  from   the classical
fixed point theorem for contractions. Denote by $\mathscr{H}_p$ the
Banach space of all the $H$-valued  continuous adapted processes $Y$ defined
on the time interval $[-\tau,T]$ such that
$Y(t)=\xi(t),t\in[-\tau,0]$, and
\begin{equation*}
\|Y\|_p:=\Big(\mathbb{E}\sup_{t\in[-\tau,T]}\|Y(t)\|^p\Big)^{\frac{1}{p}}<\infty.
\end{equation*}
Let
$$
\mathscr{K}(Y)(t)=\beg{cases} \xi(t), &\text{if}\ t\in [-\tau,0],\\
\e^{tA}\xi(0)+\int_0^t\e^{(t-s)A}F(Y_s)\d
s+\int_0^t\e^{(t-s)A}\sigma(Y(s))\d W(s), &\text{if}\  t\in (0,T].\end{cases}
$$
By (\ref{BB1}) and the linear growth of $F$ and $\si$, we conclude that   $\mathscr{K}$ maps $\mathscr{H}_p$ into
$\mathscr{H}_p$. For the existence and uniqueness of solutions, it suffices to show that the map $\mathscr{K}$ is contractive for small $T>0$.
  By the Lipschitz continuity of $F$ and $\si$, and applying (\ref{BB1}) for $\si(Y_1(s))-\si(Y_2(s))$ in place of $\si(Y(s))$,
we obtain  $$
\|\mathscr{K}(Y^1)-\mathscr{K}(Y^2)\|_p \leq CT\|Y^1-Y^2\|^p_p,\ \ Y^1,Y^2\in \scr H_p.
$$
for some constant $C>0$ and all $T\in [0,T_0].$  Choosing sufficiently small $T$ such that $CT<1$ we
can conclude that $\scr K$ is contractive.
\end{proof}

\begin{thm}
Assume that $(A1)$, $(A2)$ and $(A3)$
hold, and let $\xi,\eta\in\mathscr{C}$ and $h\in H^1_a$.
\begin{enumerate}
\item[\textmd{(1)}] $(D_hX(t))_{t\geq0}$ exists and is the
unique solution to the equation
\begin{equation*}
\begin{cases}
\d\alpha(t)=\{A\alpha(t)+\nabla_{\alpha_t}F(X^\xi_t)+\sigma(X^\xi(t))\dot{h}(t)\}\d t\\
\ \ \ \ \ \ \ \ \ \ \ \
+(\nabla_{\alpha(t)}\sigma(X^\xi(t)))\d W(t),\\
\alpha_0=0.
\end{cases}
\end{equation*}
\item[\textmd{(2)}]$ (\nabla_\eta X(t))_{t\geq0}$ exists and is the
unique solution to the equation
\begin{equation*}
\begin{cases}
\d\beta(t)=\{A\beta(t)+\nabla_{\beta_t}F(X^\xi_t)\}\d t+(\nabla_{\beta(t)}\sigma(X^\xi(t)))\d W(t),\\
\beta_0=\eta.
\end{cases}
\end{equation*}
\end{enumerate}
\end{thm}

\begin{proof}
We only prove (1) since   (2) can be proved in a   similar way. The argument of
the proof is standard in the setting of semi-linear SPDEs without
delay. The only difference for the present setting is that one has
to estimate the $\sup$ over time for the norm of the error process
for small $\epsilon\in(0,1)$
\begin{equation*}
\Lambda^{\epsilon}(t):=X^{\xi,\epsilon
h}(t)-X^\xi(t)-\epsilon\alpha(t),\ \ t\geq0,
\end{equation*}
where $X^{\xi,\epsilon h}$ is the mild solution to \eqref{eq6}.

\noindent (a) There exists a constant $C>0$ such that
\begin{equation}\label{w2}
\mathbb{E}\sup_{t\in[0,T]}\|X^{\xi,\epsilon
h}_t-X^\xi_t\|_\infty^2\leq\epsilon^2\e^{C(T+1)}\mathbb{E}\int_0^T\|\dot{h}(t)\|^2\d
t, \ \ T\geq0.
\end{equation}
Indeed, by $(A1)$, $(A2)$ and $(A3)$ we have the following It\^o's
formula for $\|X^{\xi,\epsilon h}(t)-X^\xi(t)\|^2:$
\begin{equation*}
\begin{split}
\d\|X^{\xi,\epsilon h}(t)-X^\xi(t)\|^2&=2\langle X^{\xi,\epsilon
h}(t)-X^\xi(t), A(X^{\xi,\epsilon
h}(t)-X^\xi(t))\\
&\quad+F(X^{\xi,\epsilon
h}_t)-F(X^\xi_t)+\epsilon\sigma(X^{\xi,\epsilon
h}(t))\dot{h}(t)\rangle\d t\\
&\quad+\|\sigma(X^{\xi,\epsilon h}(t))-\sigma(X^\xi(t))\|_{HS}^2\d t\\
&\quad+2\langle X^{\xi,\epsilon
h}(t)-X^\xi(t),(\sigma(X^{\xi,\epsilon h}(t))-\sigma(X^\xi(t)))\d
W(t)\rangle.
\end{split}
\end{equation*}
Noting from $(A1)$, $(A2)$ and $(A3)$ that
\begin{equation*}
\begin{split}
&\langle X^{\xi,\epsilon h}(t)-X^\xi(t), A(X^{\xi,\epsilon
h}(t)-X^\xi(t))\rangle\leq0,\\
&\|F(X^{\xi,\epsilon h}_t)-F(X^\xi_t)+\epsilon\sigma(X^{\xi,\epsilon
h}(t))\dot{h}(t)\|\leq C_1(\|X^{\xi,\epsilon
h}_t-X^\xi_t\|_\infty+\epsilon\|\dot{h}(t)\|),
\end{split}
\end{equation*}
and by the Burkhold-Davis-Gundy inequality
\begin{equation*}
\begin{split}
\mathbb{E}&\sup_{t\in[0,T]}\Big|\int_0^t\langle X^{\xi,\epsilon
h}(s)-X^\xi(s),(\sigma(X^{\xi,\epsilon
h}(s))-\sigma(X^\xi(s)))\d W(s)\rangle\Big|\\
&\leq C_1\mathbb{E}\left(\int_0^T\|X^{\xi,\epsilon
h}(s)-X^\xi(s)\|^4\d s\right)^{\frac{1}{2}}\\
&\leq\frac{1}{2}\mathbb{E}\sup_{t\in[0,T]}\|X^{\xi,\epsilon
h}(t)-X^\xi(t)\|^2+\frac{C_1}{2}\mathbb{E}\int_0^T\|X^{\xi,\epsilon
h}(s)-X^\xi(s)\|^2\d s
\end{split}
\end{equation*}
for some constant $C_1>0$, we obtain
\begin{equation*}
\begin{split}
\mathbb{E}\sup_{t\in[0,T]}\|X^{\xi,\epsilon
h}_t-X^\xi_t\|_\infty^2\leq
C_2\epsilon^2\int_0^T\|\dot{h}(t)\|^2\d
t+C_2\int_0^T\mathbb{E}\sup_{s\in[0,t]}\|X^{\xi,\epsilon
h}_s-X^\xi_s\|_\infty^2\d t
\end{split}
\end{equation*}
for some constant $C_2>0$. This implies \eqref{w2}.

\noindent (b) To prove $D_hX^{\xi}(t)=\alpha(t)$ it suffices to show
\begin{equation}\label{w3}
\lim\limits_{\epsilon\downarrow0}\frac{1}{\epsilon}\mathbb{E}\sup_{t\in[0,T]}\|\Lambda^\epsilon(t\wedge\tau_n)\|=0,
\ \ \ n\geq1,
\end{equation}
where $\tau_n:=\inf\{t\geq0,\|X_t^\xi\|_\infty\geq
n\}\uparrow\infty$ as $n\uparrow\infty$. To this end, we observe
that
\begin{equation}\label{w4}
\begin{split}
\Lambda^\epsilon(t\wedge\tau_n)&=\int_0^{t\wedge\tau_n}\e^{(t-s)A}\{F(X^{\xi,\epsilon
h}_s)-F(X^\xi_s)-\epsilon\nabla_{\alpha_s}F(X^\xi_s)\\
&\quad+\epsilon(\sigma(X^{\xi,\epsilon
h}(s))-\sigma(X^\xi(s)))\dot{h}(s)\}\d s\\
&\quad+\int_0^{t\wedge\tau_n}\e^{(t-s)A}(\sigma(X^{\xi,\epsilon
h}(s))-\sigma(X^\xi(s))-\epsilon\nabla_{\alpha(s)}\sigma(X^\xi(s)))\d
W(s).
\end{split}
\end{equation}
Let
\begin{equation*}
\gamma_n(s):=\sup_{\|\xi\|_\infty\leq n,\|\xi-\eta\|_\infty\leq
s}\|\nabla F(\xi)-\nabla F(\eta)\|_\infty+\sup_{\|x\|\leq
n,\|x-y\|\leq s}\|\nabla \sigma(x)-\nabla \sigma(y)\|_{HS}.
\end{equation*}
By $(A2)$ and $(A3)$ we have $\gamma_n(s)\downarrow0$ as
$s\downarrow0$ and $\gamma_n(\infty)<\infty$. Then
\begin{equation*}
s\gamma_n(s)\leq\gamma_n(\sqrt{\epsilon})s+\frac{s^2\gamma_n(\infty)}{\sqrt{\epsilon}},
\ \ \ \ s\geq0.
\end{equation*}
Therefore, there exists a constant $C_1>0$ such that
\begin{equation*}
\begin{split}
&\|F(X^{\xi,\epsilon
h}_s)-F(X^\xi_s)-\epsilon\nabla_{\alpha_s}F(X^\xi_s)\|_\infty\\
&\leq\|\nabla
F\|_\infty\|\Lambda_s^\epsilon\|_\infty+\|X^{\xi,\epsilon
h}_s-X^\xi_s\|_\infty\gamma_n(\|X^{\xi,\epsilon
h}_s-X^\xi_s\|_\infty)\\
&\leq
C_1\|\Lambda_s^\epsilon\|_\infty+\gamma(\sqrt{\epsilon})\|X^{\xi,\epsilon
h}_s-X^\xi_s\|_\infty+\frac{\gamma_n(\infty)}{\sqrt{\epsilon}}\|X^{\xi,\epsilon
h}_s-X^\xi_s\|_\infty^2,\\
&\epsilon\|(\sigma(X^{\xi,\epsilon
h}(s))-\sigma(X^\xi(s)))\dot{h}(s)\|\leq\epsilon^2\|\dot{h}(s)\|^2+C_1\|X^{\xi,\epsilon
h}(s)-X^\xi(s)\|^2,
\end{split}
\end{equation*}
and by the Burkhold-Davis-Gundy inequality
\begin{equation*}
\begin{split}
&\mathbb{E}\sup_{t\in[0,T]}\Big\|\int_0^{t\wedge\tau_n}\e^{(t-s)A}(\sigma(X^{\xi,\epsilon
h}(s))-\sigma(X^\xi(s))-\epsilon\nabla_{\alpha(s)}\sigma(X^\xi(s)))\d W(s)\Big\|\\
&\leq2\mathbb{E}\Big(\int_0^{T\wedge\tau_n}\|\sigma(X^{\xi,\epsilon
h}(s))-\sigma(X^\xi(s))-\epsilon\nabla_{\alpha(s)}\sigma(X^\xi(s))\|_{HS}^2\d s\Big)^{\frac{1}{2}}\\
&\leq2\mathbb{E}\Big(\int_0^{T\wedge\tau_n}(\|\nabla
\sigma\|\|\Lambda^\epsilon(s)\|+\|X^{\xi,\epsilon
h}(s)-X^\xi(s)\|\gamma_n(\|X^{\xi,\epsilon
h}(s)-X^\xi(s)\|))^2\d s\Big)^{\frac{1}{2}}\\
&\leq\frac{1}{2}\mathbb{E}\sup_{t\in[0,T]}\|\Lambda^\epsilon(t\wedge\tau_n)\|+\gamma_n(\sqrt{\epsilon})\mathbb{E}\sup_{t\in[0,T]}\|X^{\xi,\epsilon
h}(t\wedge\tau_n)-X^\xi(t\wedge\tau_n)\|\\
&\quad+\frac{\gamma_n(\infty)}{\sqrt{\epsilon}}\mathbb{E}\sup_{t\in[0,T]}\|X^{\xi,\epsilon
h}(t\wedge\tau_n)-X^\xi(t\wedge\tau_n)\|^2\\
&\quad+
C_1\mathbb{E}\int_0^{T\wedge\tau_n}\Big(\|\Lambda^\epsilon(s)\|+\gamma_n(\sqrt{\epsilon})\|X^{\xi,\epsilon
h}(s)-X^\xi(s)\|\\
&\quad+\frac{\gamma_n(\infty)}{\sqrt{\epsilon}}\|X^{\xi,\epsilon
h}(s)-X^\xi(s)\|^2\Big)\d s.
\end{split}
\end{equation*}
Combining this with \eqref{w2} and \eqref{w4} we obtain
\begin{equation*}
\begin{split}
\mathbb{E}\sup_{t\in[0,T]}\|\Lambda^\epsilon(t\wedge\tau_n)\|\leq
C_2\int_0^T\mathbb{E}\sup_{s\in[0,t]}\|\Lambda^\epsilon(s\wedge\tau_n)\|
\d
s+C(T)\Big(\gamma_n(\sqrt{\epsilon})\epsilon+\frac{\gamma_n(\infty)\epsilon^2}{\sqrt{\epsilon}}\Big)
\end{split}
\end{equation*}
for some constant $C_2>0$ and
\begin{equation*}
C(T):=\e^{C_2(1+T)}\Big(1+\mathbb{E}\int_0^T\|\dot{h}(t)\|^2\d
t\Big), \ \  T\geq0.
\end{equation*}
Due to the Gronwall inequality, this implies \eqref{w3}.

\end{proof}


\begin{thebibliography}{17}

\bibitem{ATW06}
M. Arnaudon, A. Thalmaier, F.-Y. Wang, \emph{Harnack
inequality and
  heat kernel estimates on manifolds with curvature unbounded below}, Bull.
  Sci. Math.  130 (2006),   223--233.

\bibitem{atw09} M. Arnaudon, A. Thalmaier, F.-Y. Wang, \emph{ Gradient estimates and Harnack inequalities on
non-compact Riemannian manifolds,}   Stochastic Process. Appl. {\bf
119} (2009), 3653--3670.

\bibitem{bwy11} J. Bao, F.-Y.  Wang, C. Yuan, \emph{ Derivative Formula and Harnack Inequality for Degenerate Functional
SDEs,} accessible on arXiv:1109.3907v1.

\bibitem{Bismut} J. M. Bismut, \emph{Large Deviations and the
Malliavin Calculus,} Boston: Birkh\"auser, MA, 1984.

\bibitem{DZ} G. Da Prato, J. Zabaczyk, \emph{Stochastic Equations in Infinite Dimensions,} Cambridge University Press, 1992.

\bibitem{dz96} G. Da Prato,  J. Zabczyk, \emph{Ergodicity for infinite-dimensional systems}, London Mathematical
Society Lecture Note Series, vol. 229, Cambridge University Press,
1996.



\bibitem{dx10} Z. Dong, Y. Xie, \emph{ Ergodicity of Linear SPDE
Driven by L\'{e}vy Noise,}   J. Syst. Sci. Complex  {\bf23} (2010),
137--152.


\bibitem{el94} K. D. Elworthy, X.-M. Li, \emph{Formulae for the derivatives of heat
semigroups,}   J. Funct. Anal.  {\bf125} (1994), 252--286.


\bibitem{f96} M. Fuhrman, \emph{Smoothing properties of nonlinear
stochastic equations in Hilbert spaces,}   Nonlinear Differential
Equations Appl.  {\bf 3} (1996), 445--464.

\bibitem{gw11}  A. Guillin, F.-Y. Wang, \emph{Degerate Fokker-Planck
equations: Bismut formula, gradient estimate and Harnack inequality,}
accessible on arXiv:1103.2817v2.






\bibitem{p06}  E. Priola, \emph{ Formulae for the derivatives of
degenerate diffusion semigroups,}   J. Evol. Equ.  {\bf6} (2006),
557--600.

\bibitem{RW} M. R\"ockner, F.-Y. Wang, \emph{Log-Harnack  inequality for stochastic differential equations in Hilbert spaces and its consequences, }
Infin. Dimens. Anal. Quant. Probab.  Relat. Topics {\bf13} (2010),
27--37.



\bibitem{W10} F.-Y. Wang,  \emph{Harnack  inequalities on manifolds with boundary and applications,}  J. Math. Pures Appl. {\bf94} (2010), 304--321.



\bibitem{w11} F.-Y. Wang, \emph{ Harnack inequality for SDE with multiplicative noise and
extension to Neumann semigroup on nonconvex manifolds,} Ann. Probab.
{\bf39} (2011), 1447--1467.

\bibitem{W11b} F.-Y. Wang, \emph{Derivative formula and Harnack inequality  for jump processes,} preprint,  accessible on arXiv:1104.5531.
\bibitem{wx10} F.-Y. Wang, L. Xu, \emph{ Bismut type formula and its application to stochastic hyperdissipative Navier-Stokes/Burgers equations,}
to appear in Infin. Dim. Anal. Quant. Probab. Relat. Top.,
accessible on arXiv:1009.1464.

\bibitem{wy11} F.-Y. Wang, C. Yuan, \emph{ Harnack inequalities for functional SDEs with multiplicative noise
and applications,}   Stoch. Proc. Appl. {\bf121} (2011), 2692--2710.


\bibitem{wz11}  F.-Y. Wang, X. Zhang, \emph{ Derivative Formula and
Applications for Degenerate Diffusion Semigroups,} accessible on
arXiv:1107.0096v2.

\bibitem{z10} X. Zhang, \emph{Stochastic flows and Bismut formulas for
stochastic Hamiltonian systems,}  Stoch. Proc. Appl. {\bf120}
(2010), 1929--1949.









\end{thebibliography}
\end{document}